\documentclass{article}
\usepackage{graphicx} 
\usepackage{amsthm}
\usepackage{amsfonts}
\usepackage{amsmath}
\usepackage{amssymb}
\usepackage[numbers]{natbib}
\usepackage[utf8]{inputenc}

\newtheorem{lemma}{Lemma}[section]
\newtheorem{theorem}{Theorem}[section]

\newtheorem{definition}{Definition}

\theoremstyle{remark}
\newtheorem{remark}{Remark}[section]

\title{A Topological Proof of the Archimedean Axiom for Archimedean Copulas}
\author{Victory Idowu}
\date{January 2025}

\begin{document}
\maketitle

\begin{abstract}
Archimedean copulas are a popular type of copulas in which a variant of the Archimedean axiom apply. We provide a topological proof of the Archimedean Axiom which is applicable for non-continuous distribution functions. 
\end{abstract}

\textbf{keywords}

 Copula; Archimedean axiom; Archimedean copulas;

\section{Introduction}
Archimedean copulas are a popular type of copulas in which the dependence multivariate distribution of random variables can be linked through an explicit dependence structure, determined by a parameter $\theta$ (e.g., \citep{durante2016principles, joe2014dependence, nelsen2006introduction} and the references therein). Through increasing or decreasing $\theta$, the extent of dependence between the random variables can be easily modified. 

Properties of the Archimedean copulas have been studied extensively under continuous distribution functions but limited work has been done for discrete or mixed distribution functions. The study of copula's properties for non-continuous distributions can be approached via topological arguments, like Sklar's Theorem see \citep{durante2013topological}. This is the approach taken in this paper.

\section{Copulas and the Archimedean axiom}
In this section, we provide preliminary notions on copulas and Archimedean copulas. In this paper, we denote $d \in \mathbb{Z}^+$ and $\bar{\mathbb{R}}^d := [-\infty, + \infty]^d$. 
\begin{definition}[H-volume, $V_H$ \citep{nelsen2006introduction, cherubini2009computing}]
Let $H: \bar{\mathbb{R}}^d \rightarrow \mathbb{R}$ and $[\mathbf{a}, \mathbf{b}] := [a_1, b_1] \times [a_2, b_2] \times \ldots \times [a_d, b_d]$. Then the H-volume, $V_H$ is defined as
\begin{equation*}
    V_H ([\mathbf{a}, \mathbf{b}]) = \sum_\mathbf{c} \mathrm{sign}(\mathbf{c} ) H(\mathbf{c} )
\end{equation*}
where $\mathbf{c}$ are the vertices of the box $[\mathbf{a}, \mathbf{b}]$ and 
\begin{equation*}
    \mathrm{sign}(\mathbf{c} ) = 
    \begin{cases}
    1, &\text{ if } c_k = a_k \text{ for an even number of } $k$ \text{'s} \\
    -1, &\text{ if } c_k = a_k \text{ for an odd number of } $k$ \text{'s}.
    \end{cases}    
\end{equation*}
\end{definition}

We can define the d-dimensional cumulative distribution function $F$ (d.fs) in terms of the $V_H$ volume, see for instance \citet{durante2013topological}. In this case $F$ is a distribution function if it satisfies the following definition. 
\begin{definition}[Cumulative distribution function $F$\citep{durante2013topological}]
A d-dimensional cumulative distribution function $F: \bar{\mathbb{R}}^d \rightarrow [0,1]$ satsifies the following:
\begin{enumerate}
    \item for every $\mathbf{x} \in \bar{\mathbb{R}}$ if at least one argument of $\mathbf{x}$ is 0 then $F(\mathbf{x}) = 0$ and we define $F(+ \infty, + \infty, \ldots, + \infty) = 1 $
    \item for every $j \in \{1, 2, \ldots d\} $ and for all $x_1, \ldots x_{j-1}, x_{j+1}, \ldots, x_d \in \bar{\mathbb{R}}$, $t \in \mathbb{R}$, the mapping 
    \item[] $t \mapsto F(x_1,  \ldots x_{j-1}, x_{j+1}, \ldots, x_d) $ is right continuous
    \item $V_F([\mathbf{a}, \mathbf{b}]) \geq 0$, for every box $[\mathbf{a}, \mathbf{b}]$.
\end{enumerate}    
\end{definition}

A copula is a representation of $F$ in terms of a dependence structure and the underlying distribution functions. 
\begin{definition}
A $d-$dimensional Copula $C = C(u,v)$ is the restriction to $[0,1]^d$ and $F_i \sim \mathrm{Uniform}(0,1)$ for all $i \in \{1, \ldots d\}$, see \cite{durante2013topological, kim2013normal}. 
\end{definition}

Clearly \citep{nelsen2006introduction}, $V_C ([0,u] \times [0,v]) = C(u,v)$. The uniqueness of the copula for continuous distributions is established by Sklar's theorem \citep{sklar1959fonctions}. Existence is only guaranteed for non-continuous d.fs \citep{joe2014dependence}.

\begin{definition}[Archimedean copulas \citep{joe2014dependence}]
A copula $C$ is considered as an Archimedean copula if it can be written in the form
\begin{equation*}
    C(\mathbf{u}) = \varphi \left( \sum_{j=1}^d \varphi^{[-1]} (u_j) \right), \qquad \mathbf{u} \in [0,1]^d
\end{equation*}
where generator function $\varphi$ satisfies the following:
\begin{enumerate}
  \item $\varphi: \left[ 0, 1\right] \rightarrow \left[ 0, \infty \right]$
  \item $\varphi$ is a continuous and strictly decreasing function in the interval $\left[0,1 \right]$
  \item $\varphi$ is such that $\varphi (1) = 0$.
\end{enumerate}
The pseudo inverse of the generator  is defined as:
\begin{equation}\label{phi-additive-inverse}
  \varphi^{[-1]} (x) =     \begin{cases}
                        \varphi^{-1} (x) & 0 \leq x \leq \varphi(0) \\
                        0 & \varphi (0) < x \leq \infty
                        \end{cases}
\end{equation}
\end{definition}

Any copula can have a volume defined by its recursive powers of one distribution function, this is called the $C$-power. More formally, 
\begin{definition}[$C$-power \citep{nelsen2006introduction}]
For an Archimedean copula $C_\varphi$ and any $u \in [0,1]$. Define $f_1(u) := u $ and $n+1$ $C$-power as $f_{n+1}(u) := C(u, f_n(u))$.
\end{definition}

The sub-class of Archimedean copulas are called Archimedean as their construct leads to the following version of the Archimedean axiom, \citep{nelsen2006introduction}: 
\begin{theorem} \label{c1c-thm-ArchimedeanAxiom}
The Archimedean axiom for $([0,1],C)$ is, for any two $u,v \in (0,1)$, there exists $n \in \mathbb{Z}^+$ such that $f_n(u) < v$. 
\end{theorem}
The proof of Theorem \ref{c1c-thm-ArchimedeanAxiom} for continuous d.fs is given in \citep{nelsen2006introduction}.

\begin{remark}
The Archimedean axiom is not valid at the idempotents $\{0,1\}$. When $u=1$, $f_{n+1}(u) = f_{n}(u)$, for all $n$. When $u=0$, the recursive relation $f_{n+1} = 0$ immediately stabilizes at 0 and there is no diminishing process. In both cases, the sequence $\{f_n(u)\}$ does not decrease or converge to 0 and the Archimedean axiom does not apply. 
\end{remark}

\section{Proof of Archimedean axiom}
Our goal is to show for Archimedean copulas that the Archimedean axiom also holds for non-continuous $F$. As $F$ is non-continuous, there may be regions in which $C$ and $V_C$ is also non-continuous. Hence, we develop a topological proof using $V_C$ that applies to non-continuous distributions. First, we present lemmas that will help us establish the proof. Then, we consider the limit of $V_C$ for non-continuous $F$. From which we derive how the limiting distribution of recursive distribution of $V_C$ for Archimedean copulas.  

\begin{lemma}\label{c1c-lemma-VCbounded}
For any Archimedean copula $C$, $V_C([0,u] \times [0,v]) < v$ for any $v<1$.
\end{lemma}

\begin{proof}
As the generator function $\varphi$ is strictly decreasing, $\varphi(u) > \varphi(1) = 0$ for any $u \in (0,1)$. Hence, $\varphi(u) + \varphi(v) > \varphi(u) $ which implies $\varphi^{[-1]}( \varphi(u) + \varphi(v))  < \varphi^{[-1]}(\varphi(u)) = u$. Thus, $C(u,v) < u$ when $v<1$, from which the result is immediate. \qed
\end{proof}

\begin{lemma}\label{c1c-lemma-noninc}
For  $n \geq 1, u \in(0,1)$, $f_{n+1}(u) \leq f_n(u)$.
\end{lemma}

\begin{proof}
This is shown by induction. Base case, $f_1(u) = u$ and $f_2(u) = C(u, u) = C(u, f_1(u)) \leq f_1(u)$ as for any copula $C(u,v) \leq u$ for all $v \leq u$.
Inductive step, assume $f_n(u) \leq f_{n-1}(u)$. Thus by monotonicity, 
\begin{equation*}
f_{n+1}(u) = C(u, f_n (u)) \leq C(u, f_{n-1}(u)) = f_n(u).
\end{equation*} 
Hence, $f_{n+1}(u) \leq f_n(u)$ for all $n$. \qed
\end{proof}

\section{Proof of Theorem \ref{c1c-thm-ArchimedeanAxiom}}
Suppose $C$ is has discontinuities over the box $B = [0,u] \times [0,v]$. Consider the partition of $B = \cup_{k,j} B_{k,j}$, where $C$ is continuous. Define, 
\begin{equation*}
B_k = [u_k, u_{k+1}] \times [v_j, v_{j+1}], \qquad k = 0, \ldots, m-1, \; j=0, \ldots, n-1,
\end{equation*}
where $0 = u_0 < u_1 < \ldots u_m = u$ and $0 = v_0 < v_1 < \ldots v_p = v$. For each $B_{k,j}$, 
\begin{equation*}
  V_C([u_k, u_{k+1}] \times [v_j, v_{j+1}]) =  C(u_{k+1}, v_{j+1}) - C(u_k, v_{j+1}) - C(u_{k+1}, v_j)  + C(u_k, v_{j})
\end{equation*}

As $B_k$ becomes finer, i.e. $\max(u_{k+1} - u_k, v_{j+1} - v_j) \rightarrow 0$, 
\begin{equation*}
    \lim_{m,p \rightarrow \infty} \sum_{j=0}^{m-1} \sum_{k=0}^{p-1} V_C (B_{k,j}) = V_C(B)
\end{equation*}

Convergence is guaranteed by the compactness of $[0,1]^2$ under the product topology and the continuity of $V_C$-volumes within each sub-region that $C(u,v)$ is defined. 

Define the recursive volume function as:
\begin{equation*}
    V_C^n(u) := V_C([0,u]\times [0, f_n(u)]
\end{equation*}

As $f_n(u)$ is non-decreasing, 
\begin{equation*}
    V_C^{n+1}(u) \leq  V_C^n(u)
\end{equation*}
thus, $V_C^n(u)$ is a non-increasing sequence with respect to $n$.

It follows that $f_n(u) \in [0,1]$, as $C(u,v)$ is defined on $[0,1]^2$. By Lemma \ref{c1c-lemma-noninc} $f_n(u)$ is non-increasing, it is also bounded below by $0$, it converges to a limit
\begin{equation*}
    \lim_{n \rightarrow \infty} f_n(u) = L, \qquad L \geq 0
\end{equation*}
where $L$ is the limit.  If $L >0$, then \begin{equation*}
    L = \lim_{n \rightarrow \infty}     f_n(u) = \lim_{n \rightarrow \infty} C(u, f_n (u))  = C(u,L)
\end{equation*}
However, by Lemma \ref{c1c-lemma-VCbounded}, $V_C([0,u] \times [0,v]) < v$ for $v >0$ and sufficiently large $n$. This creates a contradiction and hence, $L=0$. Hence,
\begin{equation*}
f_n(u) \rightarrow 0  \text{ and } \lim_{n \rightarrow \infty} V_C^n(u) = V_C([0,u] \times [0,0]) = 0. 
\end{equation*}

Thus there exists an $N$, for any $v \in (0,1)$, such that $f_N(u) < v$, which completes the proof. \qed

\end{document}